\documentclass{amsart}
\usepackage{bm}

\newtheorem{theorem}{Theorem}[section]

\newtheorem{corollary}[theorem]{Corollary}

\newtheorem{definition}[theorem]{Definition}
\newtheorem{example}[theorem]{Example}

\newtheorem{lemma}[theorem]{Lemma}

\newtheorem{proposition}[theorem]{Proposition}

\numberwithin{equation}{section}

\begin{document}

\title{Generalized uniform covering maps relative to subgroups}

\author{B. LaBuz}

\address{Saint Francis University, 114 Sullivan Building, Loretto, PA 15940, blabuz@@francis.edu, 814-472-3378}

\begin{abstract}
In ``Rips complexes and covers in the uniform category'' \cite{Rips} the authors define, following James \cite{J}, covering maps of uniform spaces and introduce the concept of generalized uniform covering maps. Conditions for the existence of universal uniform covering maps and generalized uniform covering maps are given. This paper extends these results by investigating the existence of these covering maps relative to subgroups of the uniform fundamental group and the fundamental group of the base space.
\end{abstract}

\maketitle
\tableofcontents

\medskip 
\medskip

\section{Introduction}

In \cite{Rips}, a theory of uniform covering maps and generalized uniform covering maps for uniform spaces is developed. In particular, it is shown that a locally uniform joinable chain connected space has a universal generalized uniform covering space and a path connected, uniformly locally path connected, and uniformly semilocally simply connected space has a universal uniform covering space. This paper extends these results to match the corresponding result in the classical setting. In this setting, one not only has a theorem about the existence of a universal covering space but a theorem about the existence of a covering space corresponding to any subgroup of the fundamental group of the base space.

A good source for basic facts about uniform spaces is \cite{B}. Let us recall some definitions and results from \cite{Rips}. Given a function $f:X\to Y$ with $X$ a uniform space, the function \textbf{generates a uniform structure} on $Y$ if the family $\{f(E):E \mathrm{\ is\ an\ entourage\ of\ } X\}$ forms a basis for a uniform structure on $Y$. If $Y$ already has a uniform structure, the function generates that structure if and only if it is uniformly continuous and the image of every entourage of $X$ is an entourage of $Y$. Given an entourage $E$ of $X$, an $\bm{E}$\textbf{-chain} in $X$ is a finite sequence $x_1,\ldots ,x_n$ such that $(x_i,x_{i+1})\in E$ for each $i\leq n$. Inverses and concatenations of $E$-chains are defined in the obvious way. $X$ is \textbf{chain connected} if for each entourage $E$ of $X$ and any $x,y\in X$ there is an $E$-chain starting at $x$ and ending at $y$. A function $f:X\to Y$ from a uniform space $X$ has \textbf{chain lifting} if for every entourage $E$ of $X$ there is an entourage $F$ of $X$ so that for any $x\in X$, any $f(F)$-chain in $Y$ starting at $f(x)$ can be lifted to an $E$-chain in $X$ starting at $x$. The function $f$ has \textbf{uniqueness of chain lifts} if for every entourage $E$ of $X$ there is an entourage $F\subset E$ so that any two $F$-chains in $X$ starting at the same point with identical images must be equal. Showing that $f$ has unique chain lifting amounts to finding an entourage of $X$ that is \textbf{transverse} to $f$. An entourage $E_0$ is transverse to $f$ if for any $(x,y)\in E_0$ with $f(x)=f(y)$, we must have $x=y$. The function has \textbf{unique chain lifting} if it has both chain lifting and uniqueness of chain lifts.

Like in the setting of paths, we wish to have homotopies of chains. Homotopies between chains were successfully defined in \cite{BP}. The following is an equivalent definition from \cite{Rips} that relies on homotopies already defined for paths. It utilizes Rips complexes which are a fundamental tool for studying chains in a uniform space. Given an entourage $E$ of $X$ the Rips complex $R(X,E)$ is the subcomplex of the full complex over $X$ whose simplices are finite $E$-bounded subsets of $X$. Any $E$-chain $x_1,\ldots ,x_n$ determines a homotopy class of paths in $R(X,E)$. Simply join successive terms $x_i,x_{i+1}$ by an edge path, i.e., a path along the edge joining $x_i$ and $x_{i+1}$. Since only homotopy classes of paths will be considered any two such paths will be equivalent. Two $E$-chains starting at the same point $x$ and ending at the same point $y$ are $\bm{E}$\textbf{-homotopic relative endpoints} if the corresponding paths in $R(X,E)$ are homotopic relative endpoints.

We wish to consider finer and finer chains in a space and therefore come to the concept of generalized paths. A \textbf{generalized path} is a collection of homotopy classes of chains $\alpha=\{[\alpha_E]\}_E$ where $E$ runs over all entourages of $X$ and for any $F\subset E$, $\alpha_F$ is $E$-homotopic relative endpoints to $\alpha_E$. Inverses and concatenations of generalized paths are defined in the obvious way. The set of generalized paths in $X$ starting at $x_0$ is denoted as $GP(X,x_0)$. We will suppress the use of a basepoint and just write $GP(X)$. $GP(X)$ is given a uniform structure generated by basic entourages defined as follows. For each entourage $E$ of $X$ let $E^*$ be the set of all pairs $(\alpha,\beta)$, $\alpha,\beta \in GP(X,x_0)$, such that $\alpha_E^{-1}\beta_E$ is $E$-homotopic to the chain $x,y$ where $x$ is the endpoint of $\alpha$ and $y$ is the endpoint of $\beta$. Call such a generalized path $\bm{E}$\textbf{-short}.

Analogously to the way the fundamental group is defined, define the group $\check\pi_1(X,x_0)$ to be the group of all generalized loops based at $x_0$. It is isomorphic to $\varprojlim (\pi_1(R(X,E),x_0))$. Also, as the notation suggests, $\check{\pi}_1(X,x_0)$ is isomorphic to the first shape group for metric compacta.

Rips complexes are used in the definition for uniform covering maps in \cite{Rips}. Let us use a formulation that is proved equivalent in that paper. A function $f:X\to Y$ is a \textbf{uniform covering map} if it generates the uniform structure on $Y$ and has unique chain lifting. 

More definitions are needed in order to define a generalized uniform covering map. Suppose $f:X\to Y$ is a function between uniform spaces. This function has \textbf{approximate uniqueness of chain lifts} if for each entourage $E$ of $X$ there is an entourage $F\subset E$ such that any two $F$-chains that start at the same point and have identical images under $f$ are $E$-close. Two chains $x_1,\ldots,x_n$ and $y_1,\ldots,y_n$ are $\bm{E}$\textbf{-close} if $(x_i,y_i)\in E$ for each $i\leq n$. The function $f$ has \textbf{generalized path lifting} if for any $x\in X$, any generalized path starting at $f(x)$ lifts to a generalized path starting at $x$.

A map $f:X\to Y$ is a \textbf{generalized uniform covering map} if it generates the uniform structure on $Y$ and has chain lifting, approximate uniqueness of chain lifts, and generalized path lifting. The definition from \cite{Rips} includes a useful property that is later proved to be redundant. Given a uniform continuous map $f:X\to Y$ there is an induced function $f_*:GP(X)\to GP(Y)$ as described in \cite{Rips}.

\begin{lemma} \cite{Rips} \label{ShortGP_Lifts}
Suppose $f:X\to Y$ is a generalized uniform covering map. Then for each entourage $E$ of $X$ there is an entourage $F$ of $Y$ so that any two generalized paths in $X$ starting at the same point are $E^*$-close if their images under $f_*$ are $F^*$-close. In particular, $F$-short generalized paths in $Y$ lift to $E$-short generalized paths in $X$.
\end{lemma}

This property can be used to show that a generalized uniform covering map has uniqueness of generalized path lifts provided $X$ is Hausdorff. This result is proved for a special case in \cite{Rips}.

\begin{proposition} \label{UniqOfGenPathLifts}
Suppose $f:X\to Y$ is a generalized uniform covering map with $X$ Hausdorff. Then $f$ has uniqueness of generalized path lifts.
\end{proposition}

\proof
Suppose two generalized paths $\alpha$ and $\beta$ in $X$ start at the same point and have $f_*(\alpha)=f_*(\beta)$. Given an entourage $E$ of $X$, by \ref{ShortGP_Lifts}, $\alpha$ and $\beta$ are $E^*$-close. Then the endpoints of $\alpha$ and $\beta$ are $E$-close. Since $X$ is Hausdorff this implies that the endpoints are equal. Then $\alpha=\beta$ since $\alpha_E$ and $\beta_E$ are $E$-homotopic rel. endpoints for each entourage $E$ of $X$.
\endproof

In fact if $Y$ is Hausdorff, then a generalized uniform covering map $f:X\to Y$ having unique generalized path lifting is equivalent to $X$ being Hausdorff.

\begin{proposition}
Suppose $f:X\to Y$ is a generalized uniform covering map with $Y$ Hausdorff. If $f$ has uniqueness of generalized path lifts then $X$ is Hausdorff.
\end{proposition}

\proof
Suppose $X$ is not Hausdorff. Then there is $x,y\in X$ with $(x,y)\in E$ for each entourage $E$ of $X$ but $x\neq y$. Consider the constant generalized path at $x$ and the generalized path $\{[x,y]_E\}_E$. They are not equal since they have different endpoints but they are both lifts of the constant generalized path at $f(x)$. Indeed, $f(x)=f(y)$ since $(f(x),f(y))\in f(E)$ for each entourage $E$ of $X$ and $f$ generates the uniform structure on $Y$.
\endproof

The most important property of covering maps is the lifting lemma, and generalized uniform covering maps have this property.

\begin{lemma} \cite{Rips} \label{LiftingLemmaForCech}
Suppose $f:X\to Y$ is a generalized uniform covering map and $g:Z\to Y$ is uniformly continuous. Suppose $X$ is Hausdorff and $Z$ is locally uniform joinable chain connected. Let $x_0\in X$, $y_0\in Y$, and $z_0\in Z$ with $f(x_0)=g(z_0)=y_0$. Then there is a unique uniformly continuous lift $h:Z\to X$ of $g$ with $h(z_0)=x_0$ if and only if $g_*(\check \pi_1(Z,z_0))\subset f_*(\check \pi_1(X,x_0))$.
\end{lemma}

Conditions for the endpoint map $GP(X,x_0)\to X$ to be a generalized uniform covering map are introduced in \cite{Rips}. A uniform space $X$ is \textbf{uniform joinable} if any two points in $X$ can be joined by a generalized path. A uniform space $X$ is \textbf{locally uniform joinable} if for each entourage $E$ of $X$ there is an entourage $F\subset E$ such that if $(x,y)\in F$, $x$ and $y$ can be joined by an $E$-short generalized path. It is easy to see that $X$ is locally uniform joinable and chain connected if and only if $X$ is locally uniform joinable and uniform joinable. We will use the former description.

\begin{proposition} \cite{Rips}
The endpoint map $GP(X,x_0)\to X$ is a generalized uniform covering map if and only if $X$ is locally uniform joinable chain connected.
\end{proposition}

In particular there is the following equivalence.

\begin{proposition} \label{GP(X)toXGenerates}
Suppose $X$ is chain connected. Then the endpoint map $GP(X,x_0)\to X$ generates the uniform structure on $X$ if and only if $X$ is locally uniform joinable.
\end{proposition}

\section{Generalized uniform covering maps relative to subgroups of the uniform fundamental group}

Let us first classify generalized uniform covering spaces over $X$ in terms of subgroups of $\check \pi_1(X)$. The corresponding theory for uniform covering spaces will follow.

We will construct, given a subgroup $H$ of $\check \pi_1(X)$, a generalized uniform covering space of $X$. It will be constructed as the quotient of an action of $H$ on $GP(X)$. Given $h\in H$ and $\alpha\in GP(X)$, let $h\cdot\alpha =h\alpha$. Let $q_H:GP(X)\to GP(X)/H$ be the induced projection. Given $\alpha\in GP(X)$, let $[\alpha]_H$ denote the orbit of $\alpha$ under the action.

Recall an action of a group $G$ on a uniform space $X$ is \textbf{uniformly equicontinuous} \cite{J} if for each entourage $E$ of $X$ there is an entourage $F$ of $X$ such that for each $g\in G$, $F\subset g^{-1}(E)$. Equivalently, $X$ has a basis of $G$-invariant entourages \cite{GroupActions}. An entourage $E$ is $\bm{G}$\textbf{-invariant} if for each $g\in G$, $gE=E$.

\begin{lemma} \label{ActionEquicontinuous}
Given a subgroup $H$ of $\check \pi_1(X)$, the action of $H$ on $GP(X)$ is uniformly equicontinuous. In particular, given an entourage $E$ of $X$, the entourage $E^*$ of $GP(X)$ is invariant under the action of $H$.
\end{lemma}

\begin{proof}
Let $E$ be an entourage of $X$. Suppose $(\alpha,\beta)\in E^*$ and $h\in H$. Then $(h\alpha)^{-1}(h\beta)=\alpha^{-1}h^{-1}h\beta=\alpha^{-1}\beta$ is $E$-short so $(h\cdot\alpha,h\cdot\alpha)\in E^*$.
\end{proof}

An action of $G$ on $X$ is \textbf{neutral} \cite{J} if for each entourage $E$ of $X$ there is an entourage $F$ of $X$ such that if $(x,gy)\in F$ there is an $h\in G$ with $(hx,y)\in E$. Since equicontinuous actions are neutral, $q_H$ has chain lifting and generates a uniform structure on $GP(X)/H$ \cite[Remark 3.2]{GroupActions}. 

\begin{proposition} \label{Composition}
Suppose $f:X\to Y$ generates the uniform structure on $Y$ and $g:Y\to Z$ is any function between uniform spaces. 
\begin{itemize}
\item[1.] If $g\circ f$ generates the uniform structure on $Z$ then $g$ generates the uniform structure on $Z$.
\item[2.] If $f$ has chain lifting and $g\circ f$ has approximate uniqueness of chain lifts then $g$ has approximate uniqueness of chain lifts.
\item[3.] If $g\circ f$ has uniqueness of chain lifts then $g$ has uniqueness of chain lifts.
\item[4.] If $g\circ f$ has chain lifting then $g$ has chain lifting.
\end{itemize}
\end{proposition}

\proof
\textrm{  }
\par\noindent
1. First let us see that $g$ is uniformly continuous. Given an entourage $E$ of $Z$, $(g\circ f)^{-1}(E)$ is an entourage of $X$. Then $f((g\circ f)^{-1}(E))$ is an entourage of $Y$. We have $g(f((g\circ f)^{-1}(E)))\subset E$. Now, given an entourage $E$ of $Y$, we wish to see that $g(E)$ is an entourage of $Z$. But $g\circ f(f^{-1}(E))\subset g(E)$ and $g\circ f(f^{-1}(E))$ is an entourage of $Z$.

\vspace{.1in}\noindent 2. Given an entourage $E$ of $Y$, choose an entourage $F\subset f^{-1}(E)$ so that two $F$-chains starting at the same point are $f^{-1}(E)$-close if their images under $g\circ f$ are identical. Choose an entourage $K\subset E$ so that $K$-chains lift to $F$-chains. Suppose two $K$-chains start at a point $y\in Y$ have identical images under $g$. Choose $x\in f^{-1}(y)$ and lift the two chains to $F$-chains in $X$ starting at $x$. These chains have identical images under $g\circ f$ so they are $f^{-1}(E)$-close. Therefore the $K$-chains are $E$-close.

\vspace{.1in}\par\noindent 3. Suppose $E_0$ is an entourage of $X$ that is transverse to $g\circ f$. Then $f(E_0)$ is an entourage of $Y$ that is transverse to $g$.

\vspace{.1in}\par\noindent 4. Given an entourage $E$ of $Y$, choose an entourage $F\subset f^{-1}(E)$ so that $g\circ f(F)$-chains lift to $f^{-1}(E)$-chains. Notice $f(F)\subset E$. Let us see that $g(f(F))$-chains in $Z$ lift to $E$-chains in $Y$. Suppose $x\in Y$ and $(g(x),y)\in g(f(F))$. Choose $x'\in f^{-1}(x)$. Then there is a $y'\in X$ with $g\circ f(y')=y$ and $(x',y')\in f^{-1}(E)$. Notice $(f(x'),f(y'))=(x,f(y'))\in E$.

\endproof

\begin{proposition} \label{p_HIsAGUCM}
Suppose $X$ is a uniform space and $H$ is a subgroup of $\check \pi_1(X)$. Then $X$ is locally uniform joinable chain connected if and only if the endpoint map $p_H:GP(X)/H\to X$ is a generalized uniform covering map.
\end{proposition}

\proof
Suppose $X$ is uniform joinable chain connected. Since $p_H\circ q_H=p$ where $p:GP(X)\to X$ is the endpoint map, according to \ref{Composition}, $p_H$ has chain lifting. Also according to \ref{Composition}, $p_H$ has approximate uniqueness of chain lifts. Finally, since $p$ has generalized path lifting, $p_H$ has generalized path lifting.

Now suppose $p_H$ is a generalized uniform covering map. In particular it generates the uniform structure on $X$ so $p=p_H\circ q_H$ generates the uniform structure on $X$. Recall $p$ generates the uniform structure on $X$ if and only if $X$ is uniform joinable chain connected (\ref{GP(X)toXGenerates}).
\endproof

We wish to have that ${p_H}_*(\check\pi_1(GP(X)/H,[\alpha_0]_H))=H$ where $\alpha_0$ is the constant generalized path at the basepoint of $X$. For this proof we need $p_H$ to have unique generalized path lifting. Recall that a generalized uniform covering map $f:X\to Y$ has unique generalized path lifting if $X$ is Hausdorff (\ref{UniqOfGenPathLifts}).

\begin{proposition} \label{GP(X)/H_Hausdorff}
Suppose $X$ is Hausdorff. Then $GP(X)/H$ is Hausdorff if and only if $H$ is a closed subgroup of $\check \pi_1(X)$.
\end{proposition}

\proof
It is easy to see that for a uniform space $X$ and a subset $A\subset X$, $x\in X$ is in the closure of $A$ if and only if $B(x,E)\cap A$ is nonempty for each entourage $E$ of $X$.

Given that $X$ is Hausdorff, $GP(X)/H$ being Hausdorff means that if $\alpha,\beta\in GP(X)$ end at the same point and for each entourage $E$ of $X$ there is a $\gamma(E)\in H$ with $(\alpha,\gamma(E)\beta)\in E^*$, then $\alpha\beta^{-1}\in H$. Now $H$ being a closed subgroup of $\check\pi_1(X)$ means that if $\lambda\in \check\pi_1(X)$ and for each entourage $E$ of $X$ there is a $\gamma(E)\in H$ with $(\lambda,\gamma(E))\in E^*$, then $\lambda\in H$. Since $E^*$ is invariant under the action of $H$, taking $\lambda=\alpha\beta^{-1}$ shows the equivalence.
\endproof

The following indicates that $H$ being closed in $\check \pi_1(X)$ is equivalent to $H$ being complete.

\begin{proposition}
$\check \pi_1(X)$ is complete.
\end{proposition}

\proof
Suppose $\mathcal F$ is a Cauchy filter in $\check \pi_1(X)$. Given an entourage $E$ of $X$, there is an $E^*$-bounded set $A_E\in\mathcal F$. Since $A_E$ is $E^*$-bounded, every element of $A_E$ has the same $E$-term, say $c_E$. Define a generalized loop $c=\{c_E\}$. To see that $c$ is a generalized path, suppose $F\subset E$ are entourages of $X$. Since $\mathcal F$ is a filter, there is an $\alpha\in A_F\cap A_E$. Then $c_E=\alpha_E$ and $c_F=\alpha_F$ so $c_F$ is $E$-homotopic to $c_E$. Finally, $\mathcal F$ converges to $c$ since $A_E\subset B(c,E^*)$. 
\endproof

\begin{corollary}
A subgroup $H$ of $\check \pi_1(X)$ is closed in $\check \pi_1(X)$ if and only if $H$ is complete.
\end{corollary}

See example \ref{ExampleNotClosed} for a space $X$ and a subgroup $H$ of $\check \pi_1(X)$ that is not complete.

\begin{proposition} \label{ImageIsH}
Suppose $X$ is Hausdorff and $H$ is a closed subgroup of $\check \pi_1(X)$. Suppose $p_H$ is a generalized uniform covering map. Then \newline $p_{H*}(\check\pi_1(GP(X)/H,[\alpha_0]_H))=H$.
\end{proposition}

\proof
By the previous proposition $GP(X)/H$ is Hausdorff. Therefore $p_H$ has unique generalized path lifting. Suppose $\alpha\in\check\pi_1(X)$. Now \newline $\alpha \in  p_{H*}(\check \pi_1(GP(X)/H,[\alpha_0]_H))$ if and only if $\alpha$ lifts to a generalized loop $\widetilde \alpha_H$ in $GP(X)/H$. In that case $\widetilde \alpha_H(1)=[\alpha_0]_H$. But $\alpha$ lifts to a generalized path $\widetilde \alpha$ in $GP(X)$ and $\widetilde \alpha (1)=\alpha$. Then $q_{H*}(\widetilde \alpha)$ is another lift of $\alpha$ in $GP(X)/H$ so $q_{H*}(\widetilde \alpha)=\widetilde \alpha_H$ and $\widetilde \alpha_H(1)=[\alpha]_H$. Therefore $[\alpha]_H=[\alpha_0]_H$ so $\alpha \in H$.
\endproof

\begin{proposition} \label{ExistenceGUCM}
Suppose a uniform space $X$ is locally uniform joinable chain connected. Then for each closed subgroup $H$ of $\check \pi_1(X)$ there is a Hausdorff, locally uniform joinable, and chain connected space $Z$ and a generalized uniform covering map $p_H:Z\to X$ such that ${p_H}_*(\check \pi_1(Z))=H$.
\end{proposition}

\proof
According to \ref{p_HIsAGUCM}, the endpoint map $p_H:GP(X)/H\to X$ is a generalized uniform covering map. Then by \ref{ImageIsH}, $p_{H*}(\check\pi_1(GP(X)/H,[\alpha_0]_H))=H$. Notice $GP(X)/H$ is Hausdorff by \ref{GP(X)/H_Hausdorff}, locally uniform joinable by \cite[Proposition 4.4]{Rips}, and chain connected by \cite[Corollary 9]{BP}.
\endproof

Recall the following theorem from the classical setting of topological covering maps: Suppose $f:X\to Y$ and $g:Z\to Y$ are two covering maps. Suppose $X$, $Y$, and $Z$ are path connected and locally path connected. Then $f$ and $g$ are equivalent if and only if the groups $f_*(\pi_1(X))$ and $g_*(\pi_1(Z))$ are conjugate in $\pi_1(Y)$. \cite{Mun} Note saying that $f$ and $g$ are equivalent means that there is a homeomorphism $h:X\to Z$ with $f=g \circ h$. We wish to have an analog of this result for generalized uniform covering maps. 

\begin{lemma} \label{ConjugateLemma1ForCech}
Suppose $f:X\to Y$  is a map between uniform spaces with $X$ uniform joinable. Let $y_0\in Y$ and $x_0,x_1\in f^{-1}(y_0)$. Then $f_*(\check \pi_1(X,x_0))$ and $f_*(\check \pi_1(X,x_1))$ are conjugate.
\end{lemma}

\proof
This lemma is proved in the same manner that it is proved in the classical setting. Let $\alpha$ be a generalized path in $X$ from $x_0$ to $x_1$. Consider the induced function $f_*:GP(X)\to GP(Y)$. First suppose $\gamma_1$ is a generalized loop in $X$ at $x_1$. Then $f_*(\alpha) f_*(\gamma_1) f_*(\alpha)^{-1}=f_*(\alpha \gamma_1 \alpha^{-1})$ so $f_*(\alpha)f_*(\check \pi_1(X,x_1))f_*(\alpha)^{-1}\subset f_*(\check \pi_1(X,x_0))$. Now suppose $\gamma_0$ is a generalized loop in $X$ at $x_0$. Then $f_*(\gamma_0)=f_*(\alpha \alpha^{-1} \gamma_0 \alpha \alpha^{-1})=f_*(\alpha) f_*(\alpha^{-1} \gamma_1 \alpha) f_*(\alpha)^{-1}$ so $f_*(\check \pi_1(X,x_0))\subset f_*(\alpha)f_*(\check \pi_1(X,x_1))f_*(\alpha)^{-1}$.
\endproof

\begin{lemma} \label{ConjugateLemma2ForCech}
Suppose a map $f:X\to Y$ between uniform spaces has unique generalized path lifting and $X$ is uniform joinable. Let $y_0\in Y$ and $x_0\in f^{-1}(y_0)$. Given a subgroup $H$ of $\check \pi_1(Y,y_0)$ that is conjugate to $f_*(\check \pi_1(X,x_0))$, there is an $x_1\in X$ so that $f_*(\check \pi_1(X,x_1))=H$.
\end{lemma}

\proof
Suppose $\alpha H \alpha^{-1}=f_*(\check \pi_1(X,x_0))$ for some $\alpha \in \check \pi_1(Y,y_0)$. Let $\widetilde \alpha$ be the lift of $\alpha$ in $X$ and $x_1$ be the endpoint of $\widetilde \alpha$. By \ref{ConjugateLemma1ForCech}, $\alpha f_*(\check \pi_1(X,x_1)) \alpha^{-1}=f_*(\check \pi_1(X,x_0))$. Therefore $f_*(\check \pi_1(X,x_1))=H$.
\endproof

We say two maps $f:X\to Y$ and $g:Z\to Y$ are uniform equivalent if there is a uniform equivalence $h:X\to Z$ with $f=g\circ h$.

\begin{proposition} \label{EquivalenceCech}
Suppose $f:X\to Y$ and $g:Z\to Y$ are two generalized uniform covering maps. Suppose $X$ and $Z$ are Hausdorff, locally uniform joinable, and chain connected. Let $x_0\in X$, $y_0\in Y$, and $z_0\in Z$ with $f(x_0)=g(z_0)=y_0$. Then $f$ and $g$ are uniform equivalent if and only if the groups $f_*(\check \pi_1(X,x_0))$ and $g_*(\check \pi_1(Z,z_0))$ are conjugate in $\check \pi_1(Y,y_0)$.
\end{proposition}

\proof
Suppose $f$ and $g$ are uniform equivalent with equivalence $h:X\to Z$ with $f=g\circ h$. Therefore $f_*(\check \pi_1(X,x_0))=g_*\circ h_*(\check \pi_1(X,x_0))=\newline g_*(\check \pi_1(Z,h(x_0)))$. By \ref{ConjugateLemma1ForCech} 
$g_*(\check \pi_1(Z,h(x_0)))$ is conjugate to $g_*(\check \pi_1(Z,z_0))$.

Now suppose the groups $f_*(\check \pi_1(X,x_0))$ and $g_*(\check \pi_1(Z,z_0))$ are conjugate. By \ref{ConjugateLemma2ForCech} there is an $x_1\in X$ with $f_*(\check \pi_1(X,x_1))=g_*(\check \pi_1(Z,z_0))$. Then by \ref{LiftingLemmaForCech} there is a lift $h:X\to Z$ of $f$ with respect to the uniform covering map $g$ and there is a lift $k:Z\to X$ of $g$ with respect to the uniform covering map $f$. Now $h\circ k$ is the identity on $X$ since it and the identity are both lifts of $f$ with respect to itself. Similarly $k\circ h$ is the identity on $Y$. Therefore $h$ is a uniform equivalence.
\endproof

We wish to classify generalized uniform covering maps of a locally uniform joinable chain connected space in terms of subgroups of its uniform fundamental group. We will only consider generalized uniform covering maps where the covering space is Hausdorff, locally uniform joinable, and chain connected.

\begin{proposition}
Suppose $f:X\to Y$ is a generalized uniform covering map. Suppose $X$ is Hausdorff. Let $x_0\in X$ and set $y_0=f(x_0)$. Then $f_*(\check\pi_1(X,x_0))$ is closed in $\check\pi_1(Y,y_0)$.
\end{proposition}

\proof
Set $H=f(\check\pi_1(X,x_0))$. Suppose $\gamma\in\mathrm{Cl}(H)$. Lift $\gamma$ to a generalized path $\tilde\gamma$ in $X$ starting at $x_0$. It suffices to show that the endpoint of $\tilde\gamma$ is $x_0$. Given an entourage $E$ of $X$, choose an entourage $F$ of $X$ so that if $\alpha,\beta\in GP(X)$ with $(f_*(\alpha),f_*(\beta))\in F^*$, then $(\alpha,\beta)\in E^*$ (see \ref{UniqOfGenPathLifts}). Now there is an $h\in H$ with $(\gamma,h)\in F^*$. Let $\tilde h\in\check\pi_1(X)$ with $f_*(\tilde h)=h$. Then $(\tilde\gamma,\tilde h)\in E^*$. In particular, $(x,x_0)\in E$ where $x$ is the endpoint of $\tilde\gamma$. Since $X$ is Hausdorff, $x=x_0$.
\endproof

\begin{theorem} \label{Classification}
Suppose $X$ is Hausdorff, locally uniform joinable, and chain connected. 
Then there is a bijective correspondence between conjugacy classes of closed subgroups of $\check\pi_1(X)$ and generalized uniform covering maps over $X$ with the covering space being Hausdorff, locally uniform joinable, and chain connected.
\end{theorem}

\proof
Let $x_0\in X$ and $f:Z\to X$ be a generalized uniform covering map where $Z$ is Hausdorff, locally uniform joinable and chain connected. We identify $f$ with the conjugacy class of the closed subgroup $H=f_*(\check\pi_1(Z,z_0))$ where $z_0\in f^{-1}(x_0)$. This identification is well defined and bijective by \ref{EquivalenceCech}. The identification is surjective by \ref{ExistenceGUCM}. 
\endproof

\section{Uniform covering maps relative to subgroups of the uniform fundamental group}

We already know that the endpoint map $\widetilde X\to X$ is a uniform covering map if and only if $X$ is path connected, uniformly locally path connected, and uniformly semilocally simply connected \cite{Rips}. Let us see when the endpoint map $GP(X)\to X$ is a uniform covering map.

\begin{definition}
A uniform space is simply uniform joinable if every generalized loop is trivial.
\end{definition}

Equivalently, a space $X$ is simply uniform joinable if $\check\pi_1(X,x_0)=1$ for each $x_0\in X$.

\begin{definition}
A uniform space is semilocally simply uniform joinable if there is an
entourage $E$ so that a generalized loop is trivial if its $E$-term is trivial.
\end{definition}

Equivalently, a space $X$ is semilocally simply uniform joinable if there is an entourage $E$ so that the projection $\check{\pi}_1(X,x_0)\to \pi_1(R(X,E),x_0)$ is a monomorphism for all $x_0\in X$.

\begin{proposition}
Suppose $f:X\to Y$ is a generalized uniform covering map with $X$ Hausdorff. If $Y$ is semilocally simply uniform joinable then $X$ is semilocally simply uniform joinable.
\end{proposition}

\proof
Suppose $Y$ is semilocally simply uniform joinable with entourage $E$ so that any generalized loop is trivial if its $E$-term is trivial. Suppose $c$ is a generalized loop at $x_{0}$ with $c_{f^{-1}(E)}$
trivial. Then $f(c)_E$ is trivial so $f(c)$ is trivial. But then $c$ and
the constant generalized path at $x_0$ are two lifts of the same
generalized path so they must be equal since $X$ is Hausdorff.
\endproof

\begin{proposition} \label{ExistenceOfUniversalUCMForCech}
The following are equivalent.
\begin{itemize}
\item[1.] The projection $\pi_X:GP(X,x_0)\to X$ is a uniform covering map for some $x_0\in X$.
\item[2.] The projection $\pi_X:GP(X,x_0)\to X$ is a uniform covering map for all $x_0\in X$.
\item[3.] $X$ is locally uniform joinable, chain connected, and semilocally simply uniform joinable.
\end{itemize}
\end{proposition}

\proof
\textrm{  }
\par\noindent
1. $\Rightarrow$ 3. Since $\pi_X$ is surjective, $X$ is uniform joinable. Since $\pi_X$ generates the structure on $X$, by \ref{GP(X)toXGenerates}, $X$ is locally uniform joinable. To see that $X$ is semilocally simply uniform joinable, let $E_0^*$ be a basic transverse entourage of $X$. Suppose $\gamma$ is a generalized loop in $X$ at a point $x\in X$ whose $E_0$-term is trivial. Let $\alpha$ be a generalized path from $x_0$ to $x$. Then $(\alpha,\alpha \gamma)\in E_0^*$ and $\pi_X(\alpha)=\pi_X(\alpha \gamma)$ so $\alpha=\alpha \gamma$ and $\gamma$ is trivial.

\vspace{.1in}\noindent 3. $\Rightarrow$ 2. Let $x_0\in X$. By \ref{GP(X)toXGenerates} $\pi_X:GP(X,x_0)\to X$ generates the structure on $X$ (uniform joinability and local uniform joinability are used here). Since $X$ is semilocally simply uniform joinable there is an entourage $E_0$ of $X$ such that any generalized loop in $X$ is trivial if its $E_0$-term is trivial. To see that $E_0^*$ is transverse to $\pi_X$, suppose $(\alpha,\beta)\in E_0^*$ with $\pi_X(\alpha)=\pi_X(\beta)$. Then $\alpha^{-1}\beta$ is a generalized loop in $X$ whose $E$-term is trivial so it is trivial. Therefore $\alpha=\beta$. Finally, to see that $\pi_X$ has chain lifting, let $E^*$ be a basic entourage of $GP(X,x_0)$ and let $F\subset E$ be an entourage of $X$ so that any two $x,y\in X$ with $(x,y)\in F$ can be joined by an $E$-short generalized path. Suppose $\alpha\in GP(X,x_0)$ and $(\pi_X(\alpha),y)\in F$. Join $\pi_X(\alpha)$ and $y$ by an $E$-short generalized path $\beta$. Then $(\alpha, \alpha \beta)\in E^*$.
\endproof

Now we wish to see that for a locally uniform joinable, chain connected, and semilocally simply uniform joinable space $X$, for each subgroup $H$ of $\check \pi_1(X,x_0)$ there is a uniform covering map $p_H:Z\to X$ with $p_*(\check \pi_1(Z,z_0))=H$.

\begin{proposition} \label{ClosedSubgroups}
If $X$ is semilocally simply uniform joinable then every subgroup $H$ of $\check\pi_1(X)$ is complete.
\end{proposition}

\proof
Suppose $\alpha$ is a generalized loop in $X$ at the basepoint of $X$ and for each entourage $E$ of $X$ there is an $h(E)\in H$ with $(\alpha,h(E))\in E^*$. There is an entourage $E_0$ of $X$ so that any generalized loop is trivial if its $E_0$-term is trivial. Therefore $\alpha h(E_0)^{-1}$ is trivial, i.e., $\alpha=h(E_0)\in H$.
\endproof

We now see an example of a space $X$ and a subgroup $H$ of $\check\pi_1(x)$ that is not complete.

\begin{example}\label{ExampleNotClosed}
According to the previous proposition, our space $X$ must be non semilocally simply uniform joinable. Consider the Hawaiian Earing $HE$. Let $HE$ be the union of circles $C_n$ of diameter $1/n$ in $\mathbb{R}^2$ with center $(0,1/2n)$ for $n\in \mathbb N$. Notice $HE$ is not semilocally uniform joinable since given a basic entourage $E_\epsilon=\{(x,y):d(x,y)\leq \epsilon\}$, if $n>1/\epsilon$, a generalized path that runs around the circle $C_n$ has trivial $E_\epsilon$ term but is not trivial. 

Given $n\in \mathbb N$, let $\alpha_n$ be the generalized path that runs around the circle $C_n$, say clockwise. Let $H$ be the subgroup of $\check\pi_1(HE)$ generated by $\{\alpha_n:n\in \mathbb N\}$. Consider the sequence of generalized paths $\beta_n=\Pi_{i=1}^n \alpha_i$. This sequence converges to a generalized path $\beta$ whose $E_\epsilon$ terms are defined to be the $E_\epsilon$ term of $\beta_n$ where $n$ is the smallest integer such that $n>1/\epsilon$. Note $\beta$ is the generalized path associated with an infinite product of paths known to be a path in $HE$.

Now $\beta\notin H$ since given any finite product of $\alpha_i$'s, there is an $E_\epsilon$ term of $\beta$ that goes around a $C_n$ not addressed by an $\alpha_i$. Since $\check\pi_1(HE)$ is Hausdorff, the limit of the sequence $\{\beta_n\}$ is unique so $H$ is not complete. Thanks to Bob Daverman for suggesting this example.
\end{example}

\begin{theorem} \label{ExistenceUCMSforGUP}
Suppose a uniform space $X$ is Hausdorff, locally uniform joinable, chain connected, and semilocally simply uniform joinable. Suppose $x_0\in X$. Then for each subgroup $H$ of $\check \pi_1(X,x_0)$ there is a Hausdorff, locally uniform joinable, and chain connected space $Z$, a uniform covering map $p_H:Z\to X$ and a point $z_0\in p_H^{-1}(x_0)$ such that $p_{H*}(\check \pi_1(Z,z_0))=H$.
\end{theorem}

\proof
Consider the endpoint map $p_H:GP(X,x_0)/H\to X$. According to \ref{ExistenceOfUniversalUCMForCech}, $\pi_X:GP(X,x_0)\to X$ is a uniform covering map. According to \ref{ActionEquicontinuous}, $q_H:GP(X,x_0)\to GP(X,x_0)/H$ generates the uniform structure on $GP(X,x_0)/H$. Then, since $\pi_X=p_H\circ q_H$, $p_H$ is a uniform covering map by \ref{Composition}. Notice $GP(X)/H$ is locally uniform joinable by \cite[Proposition 4.4]{Rips}, and chain connected by \cite[Corollary 9]{BP}. Now $H$ is closed by \ref{ClosedSubgroups}. Then $GP(X)/H$ is Hausdorff by \ref{GP(X)/H_Hausdorff}. Also, $p_{H*}(\check\pi_1(GP(X,x_0),\alpha_0))=H$ by \ref{ImageIsH}.
\endproof

We wish to classify uniform covering maps of locally uniform joinable, chain connected, and semilocally simply uniform joinable spaces. We will only consider uniform covering maps where the covering space is chain connected. Note that we do not need to explicitly assume that the covering space is Hausdorff or locally uniform joinable since these properties are inherited from the base space via uniform covering maps.

\begin{lemma}
Suppose $f:X\to Y$ is a uniform covering map.
\begin{itemize}
\item[1.] If $Y$ is Hausdorff then $X$ is Hausdorff.
\item[2.] If $Y$ is locally uniform joinable then $X$ is locally uniform joinable.
\end{itemize}
\end{lemma}

\proof
\textrm{  }
\par\noindent 1. Suppose $x,y\in X$ and $(x,y)\in E$ for each entourage $E$ of $X$. Then $(f(x),f(y))\in f(E)$ for each entourage $E$ and $f$ generates the uniform structure on $Y$ so $f(x)=f(y)$. But there is an entourage of $X$ that is transverse to $f$ so $x=y$.

\vspace{.1in}\noindent 2. Given an entourage $E$ of $X$, there is an entourage $F\subset E$ so that any two $F$-chains in $X$ starting at the same point who have identical images must be equal. By \ref{ShortGP_Lifts} there is an entourage $K$ of $Y$ so that $K$-short uniform generalized paths lift to $F$-short uniform generalized paths. Since $Y$ is locally uniform joinable there is an entourage $D$ of $Y$ so that if $(x,y)\in D$, $x$ and $y$ can be joined by an $K$-short uniform generalized path. Suppose $(x,y)\in f^{-1}(D)\cap F$. Then $(f(x),f(y))\in D$ so $f(x)$ and $f(y)$ can be joined by an $K$-short uniform generalized path. This uniform generalized path lifts to an $F$-short uniform generalized path starting at $x$. But then a representative of the $F$-term of this uniform generalized path is an $F$-chain whose image is $f(x),f(y)$ so the chain must be $x,y$.
\endproof

\begin{theorem}\label{ClassificationUnifCov}
Suppose $X$ is Hausdorff, locally uniform joinable, chain connected, and semilocally simply uniform joinable. Then there is a bijective correspondence between the the conjugacy classes of subgroups of $\check \pi_1(X)$ and uniform covering maps over $X$ with the covering space being chain connected. 
\end{theorem}

\proof
Let $x_0\in X$ and $f:Z\to X$ be a generalized uniform covering map where $Z$ is chain connected. Notice $Z$ is Hausdorff and locally uniform joinable by the previous lemma. We identify $f$ with the conjugacy class of the subgroup $H=f_*(\check\pi_1(Z,z_0))$ where $z_0\in f^{-1}(x_0)$. Notice this subgroup is closed by \ref{ClosedSubgroups}. This identification is well defined and bijective by \ref{EquivalenceCech}. The identification is surjective by \ref{ExistenceUCMSforGUP}. 
\endproof

\section{Uniform covering maps relative to subgroups of the fundamental group}

Now we consider the case of $X$ being path connected, uniformly locally path connected, and uniformly semilocally simply connected. Existence of uniform covering maps follows immediately since in this case, $GP(X,x_0)$ is uniformly equivalent to $\widetilde X$ and $\check\pi_1(X)$ is isomorphic to $\pi_1(X)$\cite{Rips}.

\begin{theorem} \label{ExistenceUP}
Suppose $X$ is Hausdorff, path connected, uniformly locally path connected, and uniformly semilocally simply connected. Suppose $x_0\in X$. Then for each subgroup $H$ of $\pi_1(X,x_0)$ there is a uniform covering map $p_H:Z\to X$ and a point $z_0\in p_H^{-1}(x_0)$ such that $p_{H*}(\pi_1(Z,z_0))=H$.
\end{theorem}

\proof
The natural map from $\widetilde X$ to $GP(X,x_0)$ is a uniform equivalence \cite{Rips} so $GP(X,x_0)\to X$ is a uniform covering map which implies that $X$ is locally uniform joinable, chain connected, and semilocally simply uniform joinable. Therefore the theorem follows from \ref{ExistenceUCMSforGUP} since $\check\pi_1(X)$ is isomorphic to $\pi_1(X)$.
\endproof

The fact that if $X$ is Hausdorff, path connected, uniformly locally path connected, and uniformly semilocally simply connected then $X$ is semilocally simply uniform joinable is nontrivial since it is not true that if $X$ is uniformly semilocally simply connected then $X$ is semilocally simply uniform joinable, even for a path connected space (see example \ref{ExampleNotSSUJ} below). Since it is nontrivial, let us include it as a proposition.

\begin{proposition} \label{UPImpliesGUP}
If $X$ is Hausdorff, path connected, uniformly locally path connected, and uniformly semilocally simply connected then $X$ is semilocally simply uniform joinable.
\end{proposition}

\begin{example} \label{ExampleNotSSUJ}
We will consider a subspace of the Hawaiian Earing to show that a path connected uniformly semilocally simply connected space need not be semilocally simply uniform joinable. Let the Hawaiian earring be as in example \ref{ExampleNotClosed}. Obtain a space $X$ by removing, for each $n$, the point of $(0,1/n)$ from circle $C_n$. Notice $X$ is path connected and uniformly semilocally simply connected (in fact $X$ is simply connected). However, $X$ is not semilocally simply uniform joinable. Given an entourage $E$ of $X$, there is an $\epsilon$ so that the basic entourage $E_\epsilon=\{(x,y):d(x,y)\leq \epsilon\}\subset E$. Then, if $n>1/\epsilon$, a generalized path that runs around the circle $C_n$ has trivial $E_\epsilon$ term (and therefore trivial $E$ term) but is not trivial.
\end{example}

Similarly, it is not true that a simply connected space is necessarily simply joinable (consider the circle with one point removed). However, we do have the following.

\begin{proposition} \label{SimplyConnectedImpliesSimplyJoinable}
If $X$ is Hausdorff, path connected, uniformly locally path connected, and simply connected then $X$ is simply uniform joinable.
\end{proposition}

\proof
If $X$ is Hausdorff, path connected, uniformly path connected, and simply connected then $X$ is uniformly equivalent to $\widetilde X$. But $X$ is path connected, uniformly locally path connected, and uniformly semilocally simply connected so $\widetilde X$ is uniformly equivalent to $GP(X)$. Since $GP(X)$ is simply uniform joinable \cite[Proposition 4.13]{Rips}, so is $X$.
\endproof

Again, we have the analog to \ref{ClassificationUnifCov} for path connected, uniformly locally path connected, and uniformly semilocally simply connected spaces.

\begin{corollary}
Suppose $X$ is Hausdorff, path connected, uniformly locally path connected, and uniformly semilocally simply connected. Then there is a bijective correspondence between the conjugacy classes of subgroups of $\pi_1(X)$ and  uniform covering maps spaces of $X$.
\end{corollary}

\proof
As above, it follows from \ref{ExistenceUP} and \ref{EquivalenceCech} since $\pi_1(X)$ is isomorphic to $\check\pi_1(X)$.
\endproof

There remains the question of when $\widetilde X\to X$ (or $\widetilde X/H\to X$ for some subgroup $H$ of $\pi_1(X)$) is a generalized uniform covering map. The following example shows that the analog of \ref{p_HIsAGUCM} need not hold. See the forthcoming \cite{TopUnifStr} for a treatment of when $\widetilde X\to X$ is a generalized uniform covering map.

\begin{example}
Consider the Hawaiian Earring $HE$ again (see \ref{ExampleNotClosed}). Let us see that $\widetilde{HE}\to HE$ does not have approximate uniqueness of chain lifts. Let the basepoint be at the origin $(0,0)$. Given any $0<\delta\leq 1/2$, we will find two $E_\delta^*$-chains in $\widetilde{HE}$ starting at the constant path at the origin with identical images in $HE$ that are not $E_{1/2}^*$-close. Choose $n$ so that $\pi/2\delta<n<\pi/\delta$. For the first chain, let $\alpha_{-1}$ and $\alpha_0$ be the constant path at the origin and for each $i\leq n$ let $\alpha_i$ be the path that starts at the origin and goes counterclockwise around $C_1$ so that it has length $i\delta$. Then $\alpha_n$ goes at least a quarter of the way around $C_1$ but less than half way. For the second chain, again let $\beta_{-1}$ be the constant path at the origin but let $\beta_0$ be the path that goes once around $C_m$ where $m$ is chosen so that $m\geq 1/\delta$. Then for each $i\leq n$ let $\beta_i=\beta_0\alpha_i$.

These two chains are $E_\delta^*$-chains and have identical images in $HE$ but $\alpha_n$ and $\beta_n$ are not $E_{1/2}^*$-close. Indeed $\alpha_n^{-1}\beta_n=\alpha_n^{-1}\beta_0\alpha_n$ cannot be homotoped to a path that is $E_{1/2}$-bounded since the distance between the endpoints of $\alpha_n$ is at least $\sqrt 2/2$.
\end{example}

\end{document}